\documentclass[a4paper,10pt]{amsart}
\usepackage{palatino, mathpazo}
\usepackage[plainpages=false]{hyperref}
\usepackage{amsfonts,latexsym,rawfonts,amsmath,amssymb,amsthm, mathrsfs, lscape}
\usepackage{verbatim}
\usepackage{stmaryrd}

\usepackage{palatino, mathpazo}
\usepackage[plainpages=false]{hyperref}
\usepackage{amsfonts,latexsym,rawfonts,amsmath,amssymb,amsthm, mathrsfs, lscape}
\usepackage[all]{xy}
\usepackage{graphicx,psfrag}

\usepackage{array, tabularx}

\usepackage{setspace}

\usepackage{epsfig}
\usepackage{graphicx}
\usepackage{eufrak}
\usepackage{dsfont}
\usepackage[english]{babel}
\usepackage{color}

\usepackage{palatino, mathpazo}
\usepackage{amscd}      
\usepackage{amssymb}
\usepackage{dsfont}
\usepackage{xypic}      
\LaTeXdiagrams          
\newbox\mybox
\def\overtag#1#2#3{\setbox\mybox\hbox{$#1$}\hbox to
  0pt{\vbox to 0pt{\vglue-#3\vglue-\ht\mybox\hbox to \wd\mybox
      {\hss$\ss#2$\hss}\vss}\hss}\box\mybox}
\def\undertag#1#2#3{\setbox\mybox\hbox{$#1$}\hbox to 0pt{\vbox to
    0pt{\vglue#3\vglue\ht\mybox\hbox to \wd\mybox
      {\hss$\ss#2$\hss}\vss}\hss}\box\mybox}
\def\lefttag#1#2#3{\hbox to 0pt{\vbox to 0pt{\vss\hbox to
      0pt{\hss$\ss#2$\hskip#3}\vss}}#1}
\def\righttag#1#2#3{\hbox to 0pt{\vbox to 0pt{\vss\hbox to
      0pt{\hskip#3$\ss#2$\hss}\vss}}#1}
\let\ss\scriptstyle

\def\Dot{\lower.2pc\hbox to 2pt{\hss$\bullet$\hss}}
\def\Circ{\lower.2pc\hbox to 2pt{\hss$\circ$\hss}}
\def\Vdots{\raise5pt\hbox{$\vdots$}}
\newcommand\lineto{\ar@{-}}
\newcommand\dashto{\ar@{--}}
\newcommand\dotto{\ar@{.}}

\allowdisplaybreaks[3]

\newtheorem{prop}{Proposition}[section]

\newtheorem{thm}[prop]{Theorem}
\newtheorem{rmk}[prop]{Remark}

\newtheorem{defn}[prop]{Definition}

\newtheorem{con}[prop]{Conjecture}
\newtheorem{condition}[prop]{Condition}

            {\nolinebreak $\Box$ \end{trivlist}}

\newcommand{\noprint}[1]{}

\renewcommand{\tilde}{\widetilde}

\newcommand{\Ext}{\mbox{Ext}}

\newcommand{\Hom}{\mbox{Hom}}

\newcommand{\tw}{\mbox{\tiny tw}}

\newcommand{\XX}{{\mathfrak X}}

\newcommand{\DD}{{\mathfrak D}}

\newcommand{\WW}{{\mathfrak W}}

\newcommand{\zz}{{\mathbb Z}}
\newcommand{\hh}{{\mathbb H}}

\newcommand{\aaa}{{\mathbb A}}

\renewcommand{\ll}{{\mathbb L}}
\newcommand{\qq}{{\mathbb Q}}
\newcommand{\pp}{{\mathbb P}}
\newcommand{\cc}{{\mathbb C}}

\newcommand{\Gm}{{{\mathbb G}_{\mbox{\tiny\rm m}}}}

\newcommand{\sT}{{\mathcal T}}
\newcommand{\sD}{{\mathcal D}}
\newcommand{\sC}{{\mathcal C}}
\newcommand{\sE}{{\mathcal E}}

\newcommand{\sL}{{\mathcal L}}

\newcommand{\sS}{{\mathcal S}}

\newcommand{\sO}{{\mathcal O}}

\newcommand{\sX}{{\mathcal X}}

\newcommand{\sM}{{\mathcal M}}

\newcommand{\sZ}{{\mathcal Z}}

\DeclareMathOperator{\Sch}{Sch}

\DeclareMathOperator{\Aut}{Aut}

\DeclareMathOperator{\ind}{ind}

\DeclareMathOperator{\vir}{vir}

\DeclareMathOperator{\Pic}{Pic}
\DeclareMathOperator{\vd}{vd}

\DeclareMathOperator{\CM}{CM}

\DeclareMathOperator{\lci}{lci}

\DeclareMathOperator{\Bub}{Bub}

\DeclareMathOperator{\KSBA}{KSBA}

\newcommand{\ev}{\mathop{\rm ev}\nolimits}

\newcommand{\ob}{\mathop{\rm ob}}

\newcommand{\spec}{\mathop{\rm Spec}\nolimits}

\numberwithin{equation}{subsection}

\newcommand {\mat}      [1] {\left(\begin{array}{#1}}
\newcommand {\rix}          {\end{array}\right)}

\begin{document}

\title{Enumerative Geometry on KSBA moduli spaces}

\author{Yunfeng Jiang}
\address{Department of Mathematics\\ University of Kansas\\ 405 Snow Hall 1460 Jayhawk Blvd\\Lawrence KS 66045 USA} 
\email{y.jiang@ku.edu}
\date{\today}
\maketitle

\begin{abstract}
We survey two new compactification methods for the KSBA moduli space of general type surfaces so that both of them admit a perfect obstruction theory.  Virtual fundamental classes exist on these two moduli spaces, and tautological invariants can be defined on KSBA moduli spaces.  This is the starting point to do enumerative geometry on KSBA moduli spaces, and we include some discussions in this direction. 
\end{abstract}

\section{Introduction}
We work over $\cc$. 
The main purpose of this survey  paper is to review the recent two compactifications of the KSBA moduli space of surfaces of general type.  These moduli spaces admit proper morphisms to the KSBA moduli space and carry a perfect obstruction theory. A virtual fundamental class can be defined on the KSBA moduli space and tautological invariants can be defined. 

\subsection{The two moduli stacks}
The KSBA moduli stack of Koll\'ar-Shepherd-Barron \cite{Kollar-Shepherd-Barron} is the higher dimension analogue of the Deligne-Mumford moduli space $\overline{M}_{g}$ of stable curves.  It is represented by the moduli functor $M_N:=\overline{M}_{K^2, \chi, N}$  which parametrizes 
$\qq$-Gorenstein deformation families $\sX/T$ of stable semi-log-canonical (slc) surfaces $X$ 
such that $\omega_{\sX/T}^{[N]}=(\omega_{\sX/T}^{\otimes N})^{\vee\vee}$ is invertible, and  $K^2:=K_X^2, \chi:=\chi(\sO_X),  N\in \zz_{>0}$, see \cite{Kollar-Shepherd-Barron}.  

Unlike the moduli space $\overline{M}_g$, where the singularities on the boundary stable curves are nodal  singularities which are locally complete intersection (lci), the boundary surfaces in $M_N$ have semi-log-canonical (slc) singularities.   SLC singularities are the most mild  singularities for the surfaces in $M_N$. 

SLC surface singularities are not lci in general.  Any slc surface $X$ can be lifted to  an index one covering DM stack $\XX\to X$ so that the only non-lci singularities of $\XX$ are slc (non log terminal) surface singularities.  More precisely, the slc and non log terminal surface singularities are simple elliptic, cusp and degenerate cusp singularities and their cyclic quotients.  When their embedded dimension $\ge 5$, they are not lci.  If an slc surface $X$ is lci, then the $\qq$-Gorenstein deformation of $X$ does not have higher obstruction spaces, i.e., the deformation of $X$ only has deformation space $\Ext^1(\Omega_{X}, \sO_X)$ and obstruction space $\Ext^2(\Omega_{X}, \sO_X)$.  But a $\qq$-Gorenstein deformation of non-lci slc singularities definitely has higher obstruction spaces, see \cite{Jiang_2021}.

\subsection*{LCI cover method} Our first method to solve this problem is to define the lci covers of these slc (non log-terminal) surface singularities.  For an slc surface $X$, there are finite  isolated  slc (non log terminal) singularities of $X$.   We prove that there is an lci covering DM stack $\XX^{\lci}\to X$ such that the DM stack $\XX^{\lci}$ has only lci singularities.  The lci covers are taken from the fundamental group of the link of the isolated singularity. 
We prove that the lci covering DM stacks can be put into flat families.  We first have 
\begin{thm}\label{thm_lci_family} (\cite[Theorem 1.4]{Jiang_2022}
    Let $\overline{f}: \sX\to T$ be a flat smoothing family of stable slc surfaces over a scheme $T$ and $f: \XX\to T$ be the associated index one covering DM stack.  Suppose that the non-lci singularity germs $(\XX_0,x)$  of the central fiber are not simple elliptic singularities of degree $5$, $6$ or $7$ or cusp singularities that can not be lifted to equivariant smoothing of an $\lci$ cusp, then $f: \XX\to T$ can be lifted to a flat family $f^{\lci}: \XX^{\lci}\to T$ of lci covering Deligne-Mumford stacks.  Moreover, $\sX$ is  the coarse moduli space of both $\XX$ and $\XX^{\lci}$.
\end{thm}

Then we define the following moduli functor 
$$
\sM_N^{\lci}:=\overline{\sM}^{\lci}_{K^2, \chi, N}: \Sch_{\mathbf{k}}\to \text{Groupoids}
$$
which sends 
$$
T\mapsto \{\XX^{\lci}\to T\}
$$
where $\{\XX^{\lci}\to T\}$ is  the isomorphic class of  flat families  of stable $\lci$ covering DM stacks. 

\begin{thm}\label{thm_lci_cover_DM_intro}(\cite[Theorem 1.6]{Jiang_2022}, \cite[Corollary 5.40]{Jiang_2022})
The moduli functor  $\sM_N^{\lci}$ is represented by a DM stack $M_N^{\lci}:=\overline{M}^{\lci}_{K^2,\chi,N}$, and there exists a proper 
morphism between DM stacks 
$$f^{\lci}: M_N^{\lci}\to M_N.$$  
\end{thm}

Not every smoothing of slc and non-log terminal singularity can be lifted to smoothing of lci covers.  The degrees $5, 6,7$ simple elliptic singularities and some cusps are very special, since the canonical singularity of their one-parameter smoothing, whose link is a $5$-dimensional Sasakian manifold and simply connected, does not admits an lci cover.  Thus, in Theorem \ref{thm_lci_cover_DM_intro}, our moduli stack $M^{\lci}_N$ can not deal with such singularities.  In \cite{Jiang_2022}, we take "crepant resolution" for the one-parameter smoothing of these singularities, and use it to define fake lci covering DM stacks.  Therefore,  the moduli stack $M_N^{\lci}$ of lci covers cover all the KSBA moduli spaces, see \cite[Theorem 1.6]{Jiang_2022}.  
When $N$ is large divisible enough,  the stack $M^{\lci}:=M_N^{\lci}$ is a proper 
DM stack and the morphism 
$f^{\lci}: M^{\lci}\to M$ is a proper  morphism which induces a proper  morphism  on their projective coarse moduli spaces.
Under fake lci covers, we get a proper moduli stack $M^{\lci}$ within algebraic geometry. 

The  idea of fake lci covering DM stacks motivates another method for the new compactification of KSBA moduli spaces. This is the bubble surface construction. 
 
\subsection*{Bubbling surface method} The motivation of bubble surfaces arises from bubbling  phenomenon of the rescaling of  K\"ahler-Einstein metrics on general-type surfaces, and these metrics collapse at log-canonical (non-log-terminal) singularities.
Similar bubbling phenomena have been used for other moduli spaces (see \cite{Feehan}, \cite{MTT}, \cite{Taubes}, \cite{Uhlenbeck}).

Let $f: \sX\to \Delta$ be a flat one-parameter smoothing family of slc surfaces. Suppose that the central fiber $f^{-1}(0)=\sX_0$ contains simple elliptic, cusp or degenerate cusp singularities $p_1, \cdots, p_n$ with local embedded dimension $\ge 6$ (since deformation of these singularities has higher obstruction spaces \cite{Jiang_2021}). 

\begin{thm}\label{thm_smoothing_bubble_intro}
(\cite[Theorem 1.1]{Jiang_Bubble})
There exists a bubble-tree surface $\sX_0^{\Bub}$ constructed from $\sX_0$ such that the simple elliptic, cusp and degenerate cusp singularities are replaced by bubble-tree simple normal crossing rational surfaces. 
Let $\XX_0^{\Bub}\to \sX_0^{\Bub}$ be the index one cover DM stack, and  
suppose that $\XX_0$ admits a smoothing.  Then the index one cover bubble tree DM stack $\XX_0^{\Bub}$ admits a smoothing $f: \mathscr{X}^{\Bub}\to \Delta$
    such that $f^{-1}(0)=\mathscr{X}_0^{\Bub}=\XX_0^{\Bub}$.

    Moreover, the smoothing $f: \mathscr{X}^{\Bub}\to \Delta$ induces smoothing of the corresponding simple elliptic, cusp or degenerate cusp singularities.
\end{thm}

From the classification of slc singularities, and the construction of index one covers, the singularity germs of the index one cover DM stack  $(\XX_{0}, p_s)$ of $\XX_0$ locally are presented by $([Z/\mu_r],p_s)$ for $r=1, 2, 3, 4, 6$, where $(Z,p_s)$ are simple elliptic, cusp and degenerate cusp singularities.  For each of these singularities in $\sX_0$,  roughly speaking we attach simple normal crossing rational surfaces along anticanonical divisors along a neighborhood of the singularity.  The simple normal crossing rational surfaces attached to cusp singularities are determined by type III degeneration of log Calabi-Yau surfaces, whose smoothing induces a smoothing of cusp singularities.  For this, we have to work on Inoue-Hirzebruch surfaces (which are complex analytic, but not algebraic surfaces), thus, the method works in complex analytic geometry.    

We  define the canonical flat families 
$f^{\Bub}: \mathscr{X}^{\Bub}\to T$ of bubble tree surfaces 
to any scheme $T$ by base change.   Different flat families 
$f^{\Bub}: \mathscr{X}^{\Bub}\to T$ of bubble tree DM stacks are related by three types of flops according to the position where the flopping rational curves $E$ meet with the singular locus of the central fiber $\mathscr{X}^{\Bub}_0$.  

Two  index one cover 
bubble-tree DM stacks $\XX^{\Bub}_1$ and $\XX^{\Bub}_2$ are called $S$-equivalent if they can be the central fibers of the extensions of the same one-parameter family $\mathscr{X}^{\Bub, \circ}\to \Delta\setminus \{0\}$ over the punctured disk to the whole disk  $\Delta$. 

Note that we do not use any lci covers for the simple elliptic, cusp and degenerate cusp singularities.  Thus, we obtain a moduli functor 
$$
M^{\Bub}_N: \Sch_{\cc}\to  (\text{Groupoids})
$$
sending a scheme $T$ to the isomorphism class  of flat families 
$f^{\Bub}: \mathscr{X}^{\Bub}\to T$ modulo $S$-equivalence.

\begin{thm}\label{thm_moduli_Bub_intro}(\cite[Theorem 1.2]{Jiang_Bubble})
    The moduli functor $M^{\Bub}_N$ is represented by a DM stack.   When $N$ is sufficiently large divisible enough, $M^{\Bub}:=M^{\Bub}_N$ is a proper DM stack with projective coarse moduli space. 

    Moreover, there exists a proper morphism 
    $f: M_N^{\Bub}\to M_N$ from the moduli stack of bubble tree index one cover DM stacks to the KSBA moduli space $M_N$. 
\end{thm}

\begin{rmk}
    In the definition of the moduli functor 
    $M_N^{\Bub}$, we can remove the $S$-equivalence condition and just consider all the deformation families of bubble-tree index one covering DM stacks. Then the moduli stack $M_N^{\Bub}$ is not separated.  But there still exists a proper map to the KSBA moduli stack $M_N^{\Bub}\to M_N$ and the virtual fundamental class can also be defined. 
\end{rmk}

These two moduli stacks should be birational.  We make the following conjecture. 

\begin{con}
    The moduli stack $M_N^{\Bub}$ can be obtained from a finite steps of  blow-up from  the KSBA moduli stack $M_N$ of index one covers along the locus such that the corresponding surfaces contains bad simple elliptic, cusp singularities with embedded dimension $\ge 6$. 

    Moreover,  the moduli stack $M^{\Bub}_N$ admits a birational  morphism 
    $M^{\Bub}_N\dasharrow M^{\lci}_N$ to the moduli stack $M^{\lci}_N$ of lci covers in \cite{Jiang_2022}, which both admit a contraction morphism to the KSBA moduli stack $M_N$, i.e., there is a diagram
    \[
    \xymatrix{
   M_N^{\Bub}\ar@{-->}[rr]\ar[dr]_{f}&& M_N^{\lci}\ar[dl]^{f^{\lci}}\\
   & M_N.&
    }
    \]
\end{con}

In \cite{ABJ}, we will show that the wall crossing between  the moduli stack of twisted stable maps to the moduli stack $\overline{M}_{1,1}$ of stable genus one curves with one marked point and the moduli stack of stable quasi-maps to $\overline{M}_{1,1}$ provides an evidence for the above conjecture.  These moduli stacks are the moduli stack of lci covers over the KSBA moduli stack of elliptic fibered surfaces.

\subsection{Virtual fundamental class}\label{subsec_pot_tautological_intro}

For an lci covering DM stack $\XX^{\lci}$,  let $T_{\XX^{\lci}}$ be the tangent sheaf. 
Roughly speaking we take  $H^1(\XX^{\lci}, T_{\XX^{\lci}})$
as the deformation space, and $H^2(\XX^{\lci}, T_{\XX^{\lci}})$ as the obstruction space.  The lci condition (plus the finite automorphisms of $\XX^{\lci}$) implies all the other $H^i(\XX^{\lci}, T_{\XX^{\lci}})$ vanish for $i\neq 1, 2$.  Then 
taking Serre duals and working in the family with the universal lci covering DM stack  we define the virtual cotangent bundle of $M_N^{\lci}$.  This construction also works for 
the moduli stack 
$M_N^{\Bub}$ since each bubble index one covering DM stack $\XX^{\Bub}$ is lci. 

Let us denote $M_N^{\KSBA}$ to be either $M_N^{\lci}$, the moduli stack of lci covers, or $M_N^{\Bub}$, the moduli stack of bubble tree index one covering DM stacks. 

The moduli stack $M_N^{\KSBA}$ is a fine moduli DM stack.  
Let 
$$p: \mathscr{X}^{\KSBA}\to M_N^{\KSBA}$$ 
be a universal family which is flat, projective and relative Gorenstein. 
Let $\ll^{\bullet}_{\mathscr{X}^{\KSBA}/M_N^{\KSBA}}$ be the relative cotangent complex of $p$ and 
$\omega^{\KSBA}:=\omega_{\mathscr{X}^{\KSBA}/M_N^{\KSBA}}[2]$. 
Let
$$E^{\bullet}_{M_N^{\KSBA}}:=Rp_{*}\left(\ll^{\bullet}_{\mathscr{X}^{\KSBA}/M_N^{\KSBA}}\otimes\omega^{\KSBA}\right)[-1],$$
where the relative dualizing sheaf $\omega^{\KSBA}=\omega_{\mathscr{X^{\KSBA}}/M_N^{\KSBA}}$  satisfies the property
$$\omega_{\mathscr{X}^{\KSBA}/M_N^{\KSBA}}|_{(p)^{-1}(t)}\cong \omega_{\mathscr{X}^{\KSBA}_t}.$$
Then the dualizing sheaf $\omega_{\mathscr{X}^{\KSBA}_t}$ of the index one cover  DM stack $\mathscr{X}^{\KSBA}_t$, which is  locally given by $\omega_{X_t}^{[r]}$  at a singularity germ ($r$ is the index of the singular germ),  
is invertible. 

The Kodaira-Spencer map $\ll^{\bullet}_{\mathscr{X}^{\KSBA}/M_N^{\KSBA}}\to p^{*}\ll_{M_N^{\KSBA}}^{\bullet}[1]$
induces a morphism 
$$\phi: E^{\bullet}_{M_N^{\KSBA}}\to \ll_{M_N^{\KSBA}}^{\bullet}.$$
In the moduli stack $M_N^{\KSBA}$, all objects in $\XX^{\KSBA}$ have only lci singularities. Let $\ll_{\XX^{\KSBA}}$ be the cotangent complex.  The space 
$\Ext^i(\ll_{\XX^{\KSBA}}, \sO_{\XX^{\KSBA}})$ is then nonzero only for  $i=1, 2$. This is because  all the singularities in the objects in $M_N^{\KSBA}$ are lci and  $\Ext^0(\ll_{\XX^{\KSBA}},\sO_{\XX^{\KSBA}})=0$ due to  the stack is of general type. Thus,  the higher obstruction spaces vanish. 
Consequently, the complex $E^{\bullet}_{M_N^{\KSBA}}$ is a two term perfect complex, and the morphism $\phi: E^{\bullet}_{M_N^{\KSBA}}\to \ll_{M_N^{\KSBA}}^{\bullet}$ defines a perfect obstruction theory in the sense of Behrend-Fantechi \cite{BF} and Li-Tian \cite{LT}. 

Different representatives in the $S$-equivalent classes give different universal families over $M_N^{\KSBA}$.  Same proof as in \cite[Theorem 6.7]{Jiang_2022} implies that two universal families $\mathscr{X}_i\to M_N^{\KSBA}$ induces equivalent perfect obstruction theories. 

Once we have a perfect obstruction theory, general theory of the intrinsic normal cone in \cite{BF}, and \cite{Jiang_2022}
implies that there is a virtual fundamental class 
$$
[M_N^{\KSBA}]^{\vir}\in A_{\vd}(M_N^{\KSBA}),
$$
where $\vd$ is the virtual dimension of $M_N^{\KSBA}$.

\begin{thm}\label{thm_virtual_class_MN}
Let $M_N=\overline{M}_{K^2,\chi,N}$ be a KSBA moduli space of general type surfaces. Then 
    the virtual fundamental class in the Chow group $A_{\vd}(M_N)$ of the KSBA space is given as the pushforward
$$
[M_N]^{\vir}=f_{*}[M_N^{\KSBA}]^{\vir}\in A_{\vd}(M_N)
$$
for the proper morphism $f: M_N^{\KSBA}\to M_N$. 
\end{thm}

Thus, we prove that there exists a virtual fundamental class on the KSBA moduli space of general type surfaces, and confirm a conjecture of Donaldson in \cite{Donaldson}.  

\subsection{Tautological invariants}\label{subsec_tautological_invariants}

Let the global index $N$ be sufficiently large divisible enough, then the moduli stack $M_N^{\KSBA}$ and  $M_N$ are both proper.

For the universal family $p: \mathscr{X}\to M_N$ of the KSBA moduli stack, the CM (Chow-Mumford) line bundle $L_{\CM}$ was proved to be ample in \cite{PXu15}.  We define the invariant 
$$
I(M_N):=\int_{[M_N]^{\vir}}(c_1(L_{\CM}))^{\vd}
$$
as the tautological invariant. 
Let $K_{\mathscr{X}/M_N}$ be the relative canonical class.  Similar to \cite{Alexeev_2023}, we can define 
$$
\kappa_i:=p_{*}(c_1(K_{\mathscr{X}/M_N})^{i+2})\in A_i(M_N).
$$
Such $\kappa_i$ are called Kappa classes, which are very important tautological classes for the moduli space $M_N$. The first Kappa class $\kappa_1$ is the first Chern class of the CM line bundle $L_{\CM}$. Thus, 
\begin{defn}\label{defn_tautological_invariants}
Let $i_1,, \cdots, i_m\in \zz_{>0}$ such that $i_1+\cdots +i_m=\vd$. Then 
we define the tautological invariants as
    $$
I_{i_1, \cdots, i_m}(M_N):=\int_{[M_N]^{\vir}}\left(\prod_{j=1}^{m}\kappa_{i_j}\right).
$$
\end{defn}

The Kappa classes were calculated for some smooth KSBA moduli space of general type surface, see \cite{Alexeev_2023}.
In \cite{Donaldson}, Donaldson conjectured that the Miller-Morita-Mumford classes (characteristic classes) of the diffeomorphism group of general type surfaces can be extended to the KSBA moduli space. 
It is thus interesting to study more general tautological classes like Hodge classes  using the virtual fundamental class $[M_N]^{\vir}$.

\subsection{Enumerative invariants on KSBA moduli spaces}

More general tautological classes can also be defined for the KSBA moduli space of stable log general type surfaces \((X, D)\), where
\[
D = \sum_{i=1}^{n} a_i D_i
\]
is a \(\mathbb{Q}\)-divisor. The KSBA moduli space of stable pairs was defined in \cite{Alexeev2}, but the perfect obstruction theory on this moduli space is quite subtle. We hope to return to the construction of the virtual fundamental class for the moduli of stable pairs. 

The moduli space of stable maps from slc surface pairs to a projective variety $W$ was constructed in \cite{Alexeev}.  There is a potential way to define enumerative invariants for the surface version of stable map spaces and answer some questions in \cite{Alexeev3}.

\subsection{Outline}
Here is the outline of the paper. 
In Section \ref{sec_lci_cover} we survey the construction of KSBA moduli space of general type surfaces, the moduli stack of index one covers and the moduli stack of lci covers.  

In Section \ref{sec_bubble_KSBA} we construct bubble tree surfaces and the moduli stack of bubble-tree index one covering DM stacks. 

In Section \ref{sec_moduli_stable_maps} we discuss the 
possibility to define Gromov-Witten like invariants counting surfaces inside projective varieties.

\subsection{Convention}

In this paper, for any DM stack $M$, the Chow ring $A_*(M)$ is always taking $\qq$ as coefficient.

We briefly review the slc nonlog terminal surface singularities. A {\em simple elliptic singularity} $(X,x)$ is a normal Gorenstein surface singularity such that the exceptional divisor of the 
minimal resolution is a smooth elliptic curve $E$. Let $E^2=-d$ for $d\in \zz_{>0}$.  
A cusp singularity is a normal Gorenstein surface singularity $(X,x)$ such that  the exceptional divisor $D$ of the minimal resolution is a cycle of smooth rational curves or a rational nodal curve.  
Let $D=\cup_i D_i$, and the self-intersection sequence is given by $(-d_1, \cdots, -d_s)$, where $d_i\ge 2$, and some $d_j\ge 3$. 
The cusp $D$ is an lci cusp if $\sum_{i}(d_i-2)\le 4$.
A {\em degenerate cusp} $(X,x)$ is a non-normal  Gorenstein surface singularity $S$. If $f: Y\to X$ is a minimal semi-resolution, then the exceptional divisor $D$ is a cycle of smooth rational curves or a rational nodal curve.  
The $Y$ has no pinch points and the irreducible components of $Y$ have cyclic quotient singularities.

\subsection*{Acknowledgments}
Y. J. would like to thank Professor Simon Donaldson for suggesting the project of virtual fundamental class for KSBA spaces, and his continuous encouragement for the study.   Y. J.  thanks Professor Richard Thomas for his   interest in this work, continuous discussion on the lci covering DM stacks, and valuable inputs on the proof of equivalence of perfect obstruction theories.  
This work is partially supported by  NSF DMS-2401484, and a Simon Collaboration Grant.

\section{Moduli space of lci covers}\label{sec_lci_cover}

In this section we survey the construction of  moduli stack of lci covering DM stacks. 

\subsection{SLC singularities}\label{subsec_slc_classification}

The classification of Koll\'ar-Shepherd-Barron on slc singularities is in  \cite[Theorem 4.24]{Kollar-Shepherd-Barron}. 
The slc surface singularities  are exactly as follows:
\begin{enumerate}
\item the semi-log-terminal singularities;
\item the Gorenstein singularities,  where every Gorenstein surface  $S$ is either semi-canonical (which is smooth, normal crossing, a pinch point or a Du Val singularity),
 or  has simple 
elliptic singularities, cusp, or degenerate cusp singularities;
\item the $\mu_2, \mu_3, \mu_4, \mu_6$ quotients of simple elliptic singularities;
\item the $\mu_2$ quotient of  cusps and degenerate cusps. 
\end{enumerate}
Here the  semi-log-terminal (slt) singularities are exactly:
\begin{enumerate}
\item the quotient of $\cc^2$ by Brieskorn \cite{Brieskorn};
\item normal crossing or pinch points;
\item $(xy=0)$ modulo the group action  given by
$x\mapsto \zeta^a x$, $y\mapsto \zeta^b y$,  and $z\mapsto \zeta z$, where $\zeta$ is a primitive 
$r$-th root of unity and $(a,r)=1, (b, r)=1$;
\item $(xy=0)$ modulo the group action 
$x\mapsto \zeta^a y$, $y\mapsto x$,  and $z\mapsto \zeta z$, where $\zeta$ is a primitive 
$r$-th root of unity and $4| r$,  $(a,r)=2$;
\item $x^2=zy^2$ modulo the group action  given by $x\mapsto \zeta^{1+a} x$, $y\mapsto \zeta^a y$,  and $z\mapsto \zeta^2 z$, where $\zeta$ is a primitive 
$r$-th root of unity and $r$ odd, and  $(a,r)=1$.
\end{enumerate}

From the  classification,  except for quotient singularity, simple elliptic, cusp, degenerate cusp singularities and their cyclic quotients, all other singularities are lci. 

\subsection{KSBA moduli space}
Let us review the KSBA moduli space of general type surfaces. 
For a projective surface  $X$, let $K_X$ be the canonical class, which is a Weil divisor class, and let $\omega_X$ be the dualizing sheaf.   For any integer $N>0$ we set 
$$\omega_X^{[N]}:=\sO_X(NK_X)=(\omega_X^{\otimes N})^{\vee\vee}.$$
\cite[Appendix to \S 1, Theorem 7]{Reid} showed that   $\omega_X$ is a torsion-free sheaf of rank one. If $X$ is normal, $\omega_X$ is a divisorial sheaf which satisfies the equivalent conditions in \cite[Appendix to \S 1, Proposition 2]{Reid}.  
In particular, $\omega_X$ is reflexive if $X$ is normal. 

\begin{defn}
    We say that $X$ has slc singularities if the following conditions hold:
\begin{enumerate}
\item the surface $X$ is reduced, Cohen-Macaulay, and has only double normal crossing singularities $(xy=0)\subset \aaa_{\mathbf{k}}^3$ away from a finite set of points;
\item  Let the pair $(X^{\nu}, \Delta^{\nu})$ be the normalization of $X$ with the inverse image of the double curve. Then $(X^{\nu}, \Delta^{\nu})$ has log canonical singularities; 
\item for some $N>0$ the $N$-th reflexive tensor power 
$\omega_X^{[N]}$ for  the dualizing sheaf $\omega_X$ is invertible. 
\end{enumerate}

A stable surface $X$ is a projective slc  surface such that $\omega_X$ is ample. 
\end{defn}

We define the KSBA moduli functor of slc surfaces by fixing  $K^2, \chi, N\in\zz_{>0}$ from \cite{Kollar-Shepherd-Barron}. 
Define 
$$M_N:=\overline{M}_{K^2, \chi, N}:  \Sch_{\cc}\to \text{Groupoids}$$
to be the moduli functor 
sending 
\begin{equation}\label{eqn_condition_functor_M}
T\mapsto  \left\{(f:\sX\to T) \left| \begin{array}{l}
  \text{$\bullet \sX\stackrel{f}{\rightarrow} T$ is a $\qq$-Gorenstein deformation family} \\
  \text{of stable s.l.c. surfaces;} \\
  \text{$\bullet$ Conditions (1)-(5) hold for each geometric fiber;} \\
   \text{$\bullet$ For each geometric point $t\in T$,  we have}\\
   \omega^{[N]}_{\sX/T}\otimes \cc(t)\to \omega_{\sX_t}^{[N]}\text{~ is an isomorphism, where}\\
   \omega^{[N]}_{\sX/T}=j_*(\omega^{\otimes N}_{\sX^0/T}),  \text{and~} j: \sX^0\to \sX
   \text{~ is the inclusion}\\
   \text{of the locus where $f$ is Gorenstein.}
     \end{array}  \right\}\right. 
\end{equation}
modulo equivalence. The Conditions (1)-(5) are given by
 \begin{enumerate}
 \item each fiber of the family $f: \sX\to T$ is a reduced projective surface;
 \item  each $\sX_t$ is connected with only s.l.c. singularities;
 \item  the sheaf $\omega_{\sX_t}^{[N]}$ which is defined by $\omega_{\sX_t}^{[N]}=j_*(\omega_{(\sX_t)^0}^{\otimes N})$ and 
 $j: (\sX_t)^0\to \sX_t$ is the inclusion of Gorenstein locus of $\sX_t$, is an ample line bundle;
 \item $K_{\sX_t}^{2}=\frac{1}{N^2}(\omega_{\sX_t}^{[N]}\cdot \omega_{\sX_t}^{[N]})=K^2$ for any $t\in T$;
 \item $\chi(\sO_{\sX_t})=\chi$ for $t\in T$.
 \end{enumerate}
 
In \cite{Kollar-Shepherd-Barron}, and \cite{KP15}, the authors proved that 
for the fixed  $K^2, \chi, N\in\zz_{>0}$,  the functor $M_N$ is represented by a  DM stack of finite type.   Suppose that  $N>0$ is large divisible enough, then the stack 
 $M_N$ is a proper DM stack with projective coarse moduli space.

\subsection{Moduli stack of index covers}

An slc surface $X$ is $\qq$-Gorebstein, not Gorenstein in general.  There is a canonical way to enhance $X$ to a Gorenstein DM stack through index one covers.  

Let $(X,x)$ be an slc singularity germ. The local index of $x\in X$ is the least integer $r$ such that $\omega_X^{[r]}$ is invertible around $x$.  Fix an isomorphism 
$\theta: \omega_X^{[r]}\to \sO_X$, we define 
$$Z:=\spec_{\sO_X}\left(\sO_S\oplus\omega^{[1]}_X\oplus\cdots\oplus \omega_X^{[r-1]}\right),$$
where the multiplication on $\sO_Z$ is defined by the isomorphism  $\theta$.  Then $\pi: Z\to X$ is a cyclic cover of degree $r$ which is called the index one cover. The morphism $\pi$ is \'etale over $X\setminus x$; and the surface $Z$ is Gorenstein.  The germ $(Z,x)$ is also slc. This is uniquely determined locally in the \'etale topology. These local index one cover DM stack $[Z/\mu_r]$ glue to give an index one covering DM stack $\XX\to X$ with $X$ the coarse moduli space.

The $\qq$-Gorenstein deformation $(x\in \sX)/(0\in T)$ is induced by an equivariant deformation of the index one cover of $(x\in X)$. 
Let $\sX/T$ be a flat family of  slc surfaces, and let $\omega_{\sX/T}$ be the  relative dualizing sheaf. 
From \cite[\S 5.4]{Kollar-Shepherd-Barron}, \cite[Appdedix to \S 1]{Reid} and \cite[\S 3.1]{Hacking}, we have 
$\omega_{\sX/T}^{[N]}:=(\omega_{\sX/T}^{\otimes N})^{\vee\vee}=i_*(\omega_{\sX^0/T}^{\otimes N})$,
where $i: \sX^{0}\hookrightarrow \sX$ is the inclusion of the Gorenstein locus.  
Let 
$\sZ/(0\in T)$ be a $\mu_r$-equivariant deformation of $Z$ inducing a $\qq$-Gorenstein deformation $\sX/(0\in T)$ of $X$, 
then  we have  
$$\sZ=\spec_{\sO_{\sX}}(\sO_{\sX}\oplus \omega^{[1]}_{\sX/T}\oplus\cdots\oplus \omega_{\sX/T}^{[r-1]}),$$
where the multiplication of $\sO_{\sZ}$ is given by fixing a trivialization of $\omega_{\sX/T}^{[r]}$.  

We can glue  these data of index one covers everywhere  locally on $\sX/T$  to form a DM stack 
$\XX/T$ which we call the index one covering DM stack associated with  $\sX/T$.  The dualizing sheaf $\omega_{\XX/T}$ is invertible. 
There is a one-to-one correspondence between the isomorphism classes $\{f:\sX\to T\}$ of  $\qq$-Gorenstein flat families of slc surfaces and the isomorphism classes $\{f:\XX\to T\}$ of index one covering DM stacks.   In \cite{AH}, \cite[\S 5]{Jiang_2022}, the moduli functor
$$
M_N^{\ind}:=\overline{M}^{\ind}_{K^2, \chi,N}: \Sch_{\cc}\to \text{Groupoids}
$$
is defined by sending 
$$
T\mapsto \{\XX\stackrel{f}{\rightarrow} T\}
$$
such that the coarse moduli space $f: \sX\to T$ of 
$f: \XX\to T$ is a $\qq$-Gorenstein deformation family satisfying the condition in (\ref{eqn_condition_functor_M}).

\begin{thm}\label{thm_index_one_DM_Moduli_intro}(\cite{AH}, \cite[Theorem 5.1]{Jiang_2022}) 
The moduli functor $M_N^{\ind}$ of index one covers  is represented by a DM stack $M_N^{\ind}$ of finite type, and it  is isomorphic to the KSBA moduli stack $M_N$. 
\end{thm}

\subsection{Lci covers}\label{subsec_lci_covers}

 In this section we review the construction of lci covers of slc non-log terminal singularities, and then define moduli stack of lci covers. We survey the construction of the three classes of slc singularities--simple elliptic singularities, cusp singularities, and degenerate cusp singularities.   More details can be found in \cite{Jiang_2022}.

\subsubsection{Simple elliptic singularities}

Let $(X,x)$ be a simple elliptic singularity of degree $d$.  This means that if  $\sigma: X\to S$ is the minimal resolution such that 
$A=\sigma^{-1}(0)$ is the exceptional elliptic curve, then  $d:=-A\cdot A$.    The local embedded dimension of the singularity is given by 
$\max(3, d)$.   It is known from \cite{Laufer}, that the simple elliptic singularity $(X,x)$ is an $\lci$ singularity if the negative self-intersection $d\le 4$.  If $d\ge 5$, then $(X,x)$ is never $\lci$. 
The link $\Sigma$ of $(X,x)$ is a $S^1$-bundle over the torus $T^2$ and $H_1(\Sigma,\zz)=\zz^2\oplus \zz_d$. The  fundamental group of the link is $\pi_1(\Sigma)=\zz^2\ltimes\zz$. 
There is an lci cover 
$$\pi: (\widetilde{X},x)\to (X,x)$$ 
for these singularities by a degree one simple elliptic singularity determined by surjective morphism $H_1(\Sigma,\zz)=\zz^2\oplus \zz_d\to \zz_d$.  The transformation group $G=\zz_d$.

From \cite{Pinkham2}, \cite{Kollar-Shepherd-Barron}, the simple elliptic singularity $(X,x)$ admits a smoothing if and only if $1\leq d\leq 9$.   
The lci cover $\pi: (\widetilde{X},x)\to (X,x)$ does not always admits an lci smoothing lifting. 
We have the following result  in \cite[Theorem 1.3]{Jiang_2023}. 
\begin{thm}\label{thm_elliptic_singularity_lci_lifting}(\cite[Theorem 1.3]{Jiang_2023})
Let $(X,x)$ be a simple elliptic surface singularity,  and $(Y, A)$ its minimal resolution. Then $(X,x)$ admits an $\lci$ smoothing lifting by a simple elliptic singularity $(\widetilde{X},0)$ of degree $\leq 4$ only when 
$d\neq 5, 6, 7$ and $1\le d\le 9$.
\end{thm}

From the classification of slc singularities in \S \ref{subsec_slc_classification}, 
the index one covering DM stack $\pi: \XX\to X$ may contain the quotient stack $[(X,x)/\mu_r]$ for $r=2,3,4,6$, where $(X,x)$ is a simple elliptic singularity of degree $d$. The $\qq$-Gorenstein smoothing of the cyclic quotient of $(X,x)$ is the same as the smoothing of the DM stack $[(X,x)/\mu_r]$. The quotient $(X,x)/\mu_r$ of $(X,x))$ is a  rational singularity, whose resolution graph of the minimal resolution is a tree.  In this case the link $\Sigma$ of $(X,x)/\mu_r$ is a rational homology sphere so that its first homology group is a finite abelian group.   We use the universal abelian cover of Newmann-Wahl to construct the lci cover. 

\begin{thm}\label{thm_minimally_elliptic_universal2}(\cite[Theorem 5.1]{Jiang_2022})
If $(X,x)$ is a $\zz_2, \zz_3, \zz_4$, or $\zz_6$ quotient of  a simple elliptic singularity,   then there exists the universal abelian cover 
$(\widetilde{X},x)$ with transformation group $G$.  Moreover, 
the $G$-equivariant  deformations of $(\widetilde{X}, x)$ gives $\qq$-Gorenstein deformations of $(X,x)$. In particular, there exists a $G$-equivariant one-parameter smoothing or deformation of $(\widetilde{X},x)$. 
\end{thm}

\subsubsection{Cusp singularities}

Let $(X,x)$ is a cusp singularity, which means that if $\pi: Y\to X$ is the minimal resolution, the resolution cycle $E$ is a circle of rational curves.
The smoothing of cusps have been studied for a long time going back to the Looijenga's conjecture in \cite{Looijenga}. 
Let us explain the idea since we use it in our lci cover construction.

Cusp singularities typically exist in pairs $(X, x)$ and $(X^\prime, x^\prime)$, which arise as the two singularities of an Inoue-Hirzebruch surface
$(\overline{V}, x, x^\prime)$
see \cite{Inoue}, and  \cite{Looijenga}. Let $D$ and $D^\prime$ be the resolution cycles of these two cusp singularities, which  consist of rational curves with sequences of negative self-intersections $d=(d_1, \cdots, d_r)$ and $d^\prime=(d_1^\prime, \cdots, d_s^\prime)$, respectively.
Each sequence must satisfy the condition $d_i \ge 2$ for all $i$, and at least one $d_j \ge 3$ to ensure the intersection matrix of $D$ is negative-definite. Furthermore, the sequence $d$ can be recovered from $d^\prime$, and vice versa. This is a consequence of the fact that the monodromy matrix of the cusp $D$ is conjugate to that of $D^\prime$.

A pair $(Y, D)$ is called a Looijenga pair (or an anticanonical pair) if $Y$ is a smooth rational surface and $D \in |-K_Y|$ is an anticanonical divisor. The smoothing of cusp singularities is closely  related to Looijenga pairs by mirror symmetry.

\begin{thm}\label{thm_Looijenga_conjecture}(\cite{Looijenga}, \cite{GHK15}, \cite{Engel})
    A cusp  $(X,x)$ admits a smoothing if and only if the dual cycle $D^\prime$ of the dual cusp $(X^\prime, x^\prime)$ lies in a rational surface $Y$ as an anticanonical divisor. 
\end{thm}

The link $\Sigma$ of the cusp singularity is a $T^2$-bundle over the circle $S^1$ and $H_1(\Sigma,\zz)=\zz\oplus G$, where $G$ is a finite abelian group. 
Let   $[d_1, \cdots, d_r]$ be the negative self-intersection sequence of $D$.   Then the monodromy of the link is given by the matrix 
$$
A=
\mat{cc}
0&-1\\
1&d_r
\rix\cdots \mat{cc}
0&-1\\
1&d_1
\rix
=
\mat{cc}
a&b\\
c&d
\rix,
$$
such that $\pi_1(\Sigma)=\zz^2\rtimes_{A}\zz$. 
From \cite[Proposition 2.5]{NW},  the cusp $(X,x)$ is an lci cusp  if and only if
$\sum_{i=1}^r(d_i-2)\le 4$, which is equivalent to the fact that  the dual cusp has
resolution cycle of length at most $4$.
From \cite[\S 4]{NW},  there is no natural epimorphism $\pi_1(\Sigma)\to G$, hence no natural Galois cover with transformation group $G$.  But in \cite[Proposition 4.1 (2)]{NW}, Neumann and Wahl  constructed a finite  cover 
$(\widetilde{X}, x)$ of $(X,x)$ with transformation group $G^\prime$ so that $(\widetilde{X}, x)$ is  a hypersurface cusp, which is  l.c.i. 

In \cite[Theorem 1.3]{Jiang_cusp}, we generalize Looijenga conjecture to the equivariant setting and prove that for any cusp singularity $(X,x)$ admitting a one-parameter smoothing, there exists an lci smoothing lifting of the singularity. 
\begin{thm}\label{thm_smoothing_lci_cusp_paper}(\cite[Theorem 1.3]{Jiang_cusp})
     Suppose that the cusp singularity $(X,x)$ admits a smoothing $f: (\sX,x)\to \Delta$.  Then there exists a smoothing $\tilde{f}: (\tilde{\sX},x)\to \Delta$ of an lci cusp together endowed with  a finite group $G$ action such that the quotient induces the smoothing $f: (\sX,x)\to \Delta$.
\end{thm}

\begin{rmk}
   In the proof of   \cite[Theorem 1.3]{Jiang_cusp}, we ahve to use Inoue-Hirzebruch surfaces, which are complex analytic surfaces and not algebraic.    
\end{rmk}

Similar to the simple elliptic singularity case, the index one covering DM stack $\pi: \XX\to X$ may contain the quotient stack $[(X,x)/\mu_2]$, where $(X,x)$ is a cusp. The quotient $(X,x)/\mu_2$ of $(X,x))$ is a  rational singularity.  So in this case the link $\Sigma$ of $(X,x)/\mu_2$ is a rational homology sphere so that its first homology group is a finite abelian group.   We still use the universal abelian cover of Newmann-Wahl to construct the lci cover. 

\begin{thm}\label{thm_minimally_elliptic_universal2}(\cite[Theorem 5.1]{Jiang_2022})
If $(X,x)$ is a $\zz_2$ quotient of  a cusp singularity,   then there exists the universal abelian cover 
$(\widetilde{X},x)$ with transformation group $G$.  Moreover, 
the $G$-equivariant  deformations of $(\widetilde{X}, x)$ gives $\qq$-Gorenstein deformations of $(X,x)$.
\end{thm}

\subsubsection{Degenerate cusp singularities}

Let $(X,x)$ be a degenerate cusp singularity, which is a non-normal surface singularity such that if 
$$\pi: \widetilde{X}\to X$$
is the normalization, then each component of $\widetilde{X}$ is a cyclic quotient singularity which is a rational singularity. 
The quotient $(X,x)/\mu_2$ is also slc, and the stack $[(X,x)/\mu_2]$ can be a chart of the index one covering DM stack $\XX\to X$. It is known that any degenerate cusp singularity admits a smoothing. 

For these singularities, we use the universal abelian cover of Newmann-Wahl \cite{NW} to construct the lci cover. 

\begin{thm}\label{thm_minimally_elliptic_universal}(\cite[Theorem 5.8]{Jiang_2022})
Let $(X,x)$ be  a degenerate cusp singularity  germ. Then there is an lci cover 
$(\widetilde{X},x)$ of  $(X,x)$ with transformation group $G$.    
The $G$-equivariant  deformations of $(\widetilde{X}, x)$ induce Gorenstein deformations of  $(X,x)$.
\end{thm}

\subsection{Moduli stack of lci covers}

Recall that a simple elliptic singularity $(X,x)$ of degree $6,7$ does not admits an lci smoothing lifting.  For such singularities, let
$(\sX,x)\to (\Delta,0)$ be a  $\qq$-Gorenstein family.  Then we take crepant resolution to get a one-parameter flat family   $\widetilde{\pi}: (\widetilde{\sX}, 0)\to (\Delta,0)$ of $\lci$ surface  such that all the fibers of the index one covering DM stack of the morphism $\widetilde{\pi}: (\widetilde{\sX}, 0)\to (\Delta,0)$ have lci singularities and we call the central fiber $\widetilde{\sX}_0$ a fake lci covering DM stack.

Two  smoothing families of lci covering DM stacks using crepant resolutions are related by flops.  We borrow a definition from K-semistable moduli space and  define $\mathbb{S}$-equivalent relations.
Two  fake lci covering DM stacks $\XX^{\lci}_1$ and $\XX^{\lci}_2$ are called $\mathbb{S}$-equivalent if a flat  family $\XX^{\circ, \lci}\to T\setminus\{0\}$ over $T\setminus\{0\}$ can be extended to families over $T$ by adding $\XX^{\lci}_1$ and $\XX^{\lci}_2$.  

We summarize the one-parameter smoothing  families in 
Theorem \ref{thm_minimally_elliptic_universal2}, Theorem \ref{thm_smoothing_lci_cusp_paper}, Theorem \ref{thm_elliptic_singularity_lci_lifting}, and 
 the above analysis 
with  the following result.

\begin{thm}\label{thm_one_parameter_family_lci}
All the one-parameter smoothing and deformation families
$f: \sX\to \Delta$
of s.l.c. surfaces  can be obtained by smoothing and deformation families $f^{\lci}: \XX^{\lci}\to \Delta$ of fake $\lci$ covering DM stacks. 
\end{thm}

We then define the flat families over any base scheme $T$ by base change. 

\begin{defn}\label{defn_family_lci}
We define the flat  families over a scheme $T$ in the following diagram
\begin{equation}\label{eqn_diagram_ind_lci_family}
\xymatrix{
\XX^{\lci}\ar[rr]^{\hat{\pi}}\ar[dr]^{\pi^{\lci}}\ar[dd]_{f^{\lci}}&&\XX\ar[dl]_{\pi}\ar[ddl]^{f}\\
&\sX\ar[d]_-{\overline{f}}&\\
T\ar[r]^{\eta}&T^\prime&
}
\end{equation} 
where 
\begin{enumerate}
\item  $\overline{f}: \sX\to T^\prime$ is a $\qq$-Gorenstein deformation  family of s.l.c. surfaces;
\item $f: \XX\to T^\prime$ is the corresponding index one covering DM stack;
\item $f^{\lci}: \XX^{\lci}\to T$ is the lifting  fake $\lci$ covering DM stack of $\overline{f}$, such that the morphism $\pi^{\lci}: \XX^{\lci}\to \sS$ factors through the morphism 
$\pi: \XX\to \sX$.  The morphism $\eta: T\to T^\prime$ is proper;
\item whenever there is a one-parameter family $\sX\to \Delta^\prime$ such that $\Delta^\prime\to T^\prime$, then up to finite cover $\Delta\to \Delta^\prime$, we have an lci lifting $\XX^{\lci}\to \Delta$ such that $\Delta\subset T$;
\item the isomorphic classes   $\{\overline{f}: \sX\to T^\prime\}$  of the  families   must satisfy the 
conditions in (\ref{eqn_condition_functor_M}).
\end{enumerate}
\end{defn}

Then we have

\begin{thm}\label{thm_universal_covering_stack} (\cite[Theorem 5.37]{Jiang_2022})
The functor $\sM_N^{\lci}$ represents a DM stack $M_N^{\lci}:=\overline{M}^{\lci}_{K^2,\chi,N}$.  Moreover,  there exists a proper morphism 
$$f^{\lci}: M_N^{\lci}\to M_N.$$ 

In particular,  if $N$ is large divisible enough, the stack $M^{\lci}:=M_N^{\lci}$ is a proper DM stack with projective coarse moduli space. 
The morphism $f^{\lci}$ in the above diagram induces a proper  morphism on their coarse moduli spaces.
\end{thm}

\section{Bubble compactification of KSBA spaces}\label{sec_bubble_KSBA}

In this section we review the bubble tree compactification of KSBA moduli space of general type surfaces.  This is motivated from the collapsing of K\"ahler-Einstein metrics on general type surfaces. 

\subsection{Bubble surfaces}\label{subsec_bubble_surfaces}

We survey the construction of bubble tree surfaces, and bubble tree index one covering DM stacks. The idea is to replace the simple elliptic singularity, cusp singularity and degenerate cusp singularity  in an slc surface $X$ by simple normal crossing rational surfaces gluing along their anticanonical divisors. 
This is similar to the stable curves in $\overline{M}_g$, where the boundary curves are nodal curves containing components of rational details.  Thus, it is reasonable to take the bubble tree surfaces as 
the 2-dimensional analogue of stable curves. 

We will survey the basic idea of the construction for the three types of singularities--the simple elliptic singularity, the cusp singularity and the degenerate cusp singularity in an slc surface $X$. 

Suppose that there is a sequence $(X_j, \omega_j, p_j)$ of K\"ahler-Einstein metrics which converges to a projective surface $(X_{\infty}, p_{\infty})$ with log-canonical singularities. 
It is expected that the K\"ahler-Einstein metrics will collapse around the log-canonical singularity $p_{\infty}$.  
The Gromov-Hausdorff limit of the sequence
$(X_j, \omega_j, p_j)$ is the complement of the singularity 
$p_{\infty}$ in $X_{\infty}$. 
Thus, our construction of the bubble tree surfaces should be compatible with the one-parameter smoothing of these singularities. 

\subsubsection{Simple elliptic singularities}

Let $(X,x)$ be a simple elliptic singularity, and $f: \sX\to \Delta$ be a one-parameter smoothing such that $(\sX_0,x)=(X,x)$.
Then the degree $d$ of $(X,x)$ satisfies $1\le d\le 9$.

For the simple elliptic singularities $(\sX_0,x)$,  we
replace the singularities $x$ by bubble tree of rational surfaces. We explain the construction here. 
Let $(\widetilde{X}, E)\to (X,x)$ be the minimal resolution such that $E$ is a smooth elliptic curve such that $E^2=-d$.   Let $(X_i, E_i)$ for $i=1,2, \cdot, m$ be a finite number of rational surfaces so that $X_i$ are rational fibrations over smooth elliptic curve $E$ for $1\le i\le m-1$; and 
$(X_m, E_m)$ be the smooth del Pezzo surface  with degree $d$ and the smooth elliptic curve $E_m\in |-K_{X_m}|$.

We let  
\begin{equation}\label{eqn_simple_E}
    X^{\Bub}:=\widetilde{X}\bigcup_{E}\bigcup_{i=1}^m X_{i}.
\end{equation}
where the gluing is along the elliptic curves $E$.  The open affine surface $Y\setminus E$ is  the bubble limit of the rescaling of  K\"ahler-Einstein metrics; see \cite{FHJ}. 
 This is similar to the type II degeneration of K3 surfaces, although type II degeneration of K3 surfaces  is a chain of rational surfaces.

If the simple elliptic singularity $(X,x)$ admits a cyclic group $\mu_r$ (for $r=2,,3,4,6$) action,  it induces an action  on the chain of  rational surfaces.  
When we construct the index one cover DM stacks, this is taken as the quotient stack of the chain of rational surfaces. 

\subsubsection{Cusp singularities}

Let $(X,x)$ be a cusp singularity, and let $f: (\sX,x)\to (\Delta,0)$ be a smoothing so that $(\sX_0,x)=(X.x)$.

For this cusp singularity $(\sX_0,x)$, we use Looijenga's conjecture (proved in \cite{GHK15}, \cite{Engel}) to replace the singularity by   a type III degeneration of log Calabi–Yau surfaces. Therefore, the mirror symmetry properties between smoothing cusps and Looijenga pairs associated to dual cusp cycles (which is Looijenga's conjecture) play a key role in the construction.

Since $(X,x)$ admits a smoothing, there is a Looijenga pair $(Y,D^\prime)$ associated with the dual cycle of the dual cusp. 
P. Engel employed an approach involving Type III degeneration of Looijenga pairs, analogous to the Type III degeneration used for K3 surfaces. From the pair $(Y, D^\prime)$, he constructed a Type III degeneration of Looijenga pairs, which is a normal crossing variety
\begin{equation}\label{eqn_XX_0}
    \mathscr{X}_0=\bigcup_{i=0}^{m-1}V_i
\end{equation}
where $(V_0, D, D^\prime)$ is
the Inoue surface with the resolution cycle $D, D^\prime$ of the two dual cusps; and for $i>0$, $V_i$ are rational surfaces such that $(V_i, D_i)$ is a Looijenga pair. Here we have 
$$
D=D_0=\bigcup_{i}D_{0i},  
$$
where $D_{0i}=V_0\cap V_i$, and for $i>0$, 
$D_i=\cup D_{ij}$, $D_{ij}=V_i\cap V_j$ are double curves. 
Moreover, the dual complex $\Gamma(\mathscr{X}_0)$ is a triangulation of sphere.  For each triangle, 
the 
triple point formula   
\begin{equation}\label{eqn_triple_point}
d^{\prime}_{ij}:=
\begin{cases}
-D^{\prime 2}_{ij}, & \ell(D^{\prime}_i)\ge 2;\\
2-D^{\prime 2}_{ij}, & \ell(D^{\prime}_i)=1
\end{cases}
\end{equation}
holds, where  $d^{\prime}_{ij}$ is the negative self-intersection of $D^\prime_{ij}$.

Let 
$$
X^{\Bub}:=X\bigcup_{D} \left(\mathscr{X}_0\setminus V_0\right)
$$
where the gluing is along the neighborhood of $D\subset X$ and $D\subset \mathscr{X}_0$.

 The sequence of K\"ahler–Einstein metrics collapses at the cusp singularity. In \cite[Corollary 1.11]{DFS}, the authors proved that any complete K\"ahler–Einstein metric near the cusp singularity 
$x$ is asymptotically isometric to $\hh\times \hh/\Gamma$, where 
$\hh$
is the upper half-plane and 
$\Gamma$ is a discrete parabolic subgroup of $\Aut(\hh\times\hh)$. The space 
$\hh\times \hh/\Gamma$ is a non-compact manifold with a cusp end. We hope  that the type III degeneration is the bubble limits of K\"ahler–Einstein metrics.

\begin{prop}\label{prop_smoothing_X_E}
    The bubble surface $X^{\Bub}$ in both the case of simple elliptic singualrity and the cusp singularity case  admits a smoothing so that it induces the smoothing of the corresponding singularities.  
\end{prop}
\begin{proof}
In the case of simple elliptic singularity case, 
$X^{\Bub}$ admits a smoothing 
$$
f^{\Bub}: \mathscr{X}^{\Bub}\to \Delta
$$
such that 
the rational surface  components in $\mathscr{X}_0$ are elliptic ruled surfaces and a smooth del Pezzo surface of degree $d$.  They contract to the smoothing of the original simple elliptic singularity. 

 In the cusp singularity case,    \cite[Theorem 5.2]{Engel} proved that $\mathscr{X}_0$ admits a smoothing $f: \mathscr{X}\to \Delta$ which gives the smoothing of the cusp $D$. The generic fiber of $f$ is the Looijenga pair $(Y,D^\prime)$.
More precisely, the smoothing of the cusp $D$ is given by the smoothing $f: \mathscr{X}\to \Delta$ as follows.  From \cite{Shepherd-Barron}, all the rational surfaces $V_i$ for $i>0$ in the central fiber 
$f^{-1}(0)=\mathscr{X}_0$ are contractible which gives the cusp $D$.
Therefore, all of the rational components $V_i$ for $i>0$ contract to give the smoothing of the original cusp. 
\end{proof}

\subsubsection{Degenerate cusp singularities}

Let $(X,x)$ be a degenerate cusp singularity.  It always admits a smoothing
$$
f: (\sX,x)\to(\Delta,0)
$$
such that $(\sX_0,x)=(X,x)$.
In this case  we use crepant resolution to form bubble surfaces. Look at the diagram
\[
\xymatrix{
(\widetilde{\sX},x)\ar[rr]^{\pi}\ar[dr]_{\tilde{f}}&&(\sX,x)\ar[dl]^{f}\\
&(\Delta,0)&
}
\]
where $\pi: (\widetilde{\sX},x)\to (\sX,x)$ is a crepant resolution. 
The central fiber of $\tilde{f}$ is an open Kulikov surface whose mirror symmetry property   was used in \cite{AAB2024} 
to construct KSBA families of  log Calabi-Yau surfaces.  The open Kulikov surface is attached to the neighborhood of the degenerate cusp singularity to form bubble surfaces. 
We hope that this open Kulikov surface is the bubbling limit structure  in \cite{Sun}.

\subsection{Bubble compactification of KSBA moduli spaces}\label{subsec_moduli_Bubble}

Based on the above construction we construct the bubble compactification of KSBA moduli space of general type surfaces. 

Let $f^{\Bub}: \mathscr{X}^{\Bub}\to \Delta$ be a one-parameter flat family of bubble tree surfaces in above section. We first take the index one covering DM stack
$$
f^{\Bub}: \XX^{\Bub}\to \Delta
$$
of the family  $f^{\Bub}: \mathscr{X}^{\Bub}\to \Delta$ so that each fiber of this family of index one covering DM stacks has only lci singularities. 
We then  define the flat family of bubble tree index one covering DM stacks by base change. 

\begin{defn}\label{defn_flat_family}
    A flat family $f: \XX^{\Bub}\to T$ of bubble-tree index one covering  DM stacks over an arbitrary scheme $T$ is a flat family over $T$ such that whenever there is a DVR $R$ we have the following Cartesian diagram
    \[
    \xymatrix{
    \XX\ar[r]\ar[d]& \XX^{\Bub}\ar[d]\\
    \spec(R)\ar[r]& T
    }
    \]
    where $f: \XX\to \spec(R)$ is the one-parameter flat family of bubble-tree index one covering DM stacks. We call the family $f: \XX^{\Bub}\to T$ a canonical family of index one cover bubble-tree DM stacks. 
\end{defn}

Let us introduce the flat families of bubble-tree index one covering DM stacks in the moduli functor. 
\begin{defn}\label{defn_moduli_functor}
    Let $T$ be a scheme, we define the following diagram of flat families of geometric objects:
    \begin{equation}\label{eqn_diagram}
        \xymatrix{
        \XX^{\Bub}_{\min}\ar@{-->}[r]^{\varphi}\ar[d]_{f}& \XX\ar[d]^{f^\prime}\ar[r]& \sX\ar[dl]^{f_{\sX}}\\
        T\ar[r]^{\phi}& \overline{T} &
        }
    \end{equation}
    where 
    \begin{enumerate}
        \item $f^\prime: \XX\to \overline{T}$ is a flat family of index one covering DM stacks, and $f_{\sX}: \sX\to \overline{T}$ is the corresponding flat $\qq$-Gorenstein deformation family of slc stable surfaces. 
        \item $\{f: \XX^{\Bub}_{\min}\to T\}$ is  the isomorphism  class of  flat families of canonical index one cover bubble-tree DM stacks. 
        \item $\phi: T\to \overline{T}$ is a proper map.
        \item the dash arrow $\varphi:  \XX^{\Bub}\dasharrow \XX$ means that we do not require that there is a direct morphism, 
        but induces a flat family. 
        \item if two flat families $f_1: \XX_1^{\Bub}\to T$ and $f_2: \XX_2^{\Bub}\to T$ differ only on the bubbling rational surfaces, then they are related by a flop.
    \end{enumerate}
\end{defn}

    Let $\Sch_{\cc}$ be the category of schemes over $\cc$.  We define the functor 
    $$
    M^{\Bub}_{N}:=\overline{M}^{\Bub}_{K^2, \chi, N}: \Sch_{\cc}\to \text{Groupoids}
    $$
    such that 
    $$
    T\mapsto \{f: \XX^{\Bub}_{\min}\to T\}
    $$
    where  $\{f: \XX^{\Bub}_{\min}\to T\}$ is the isomorphism class of flat families of canonical bubble-tree index one cover DM stacks. 
    Moreover, the coarse moduli space $f_{\sX}: \sX\to \overline{T}$ satisfies the numerical condition the KSBA moduli space $M_N$ in (\ref{eqn_condition_functor_M}) and Conditions (1)-(5) in \S 2.1. 

\begin{thm}\label{thm_moduli_stack}(\cite[Theorem 5.12]{Jiang_Bubble})
    The moduli functor $M_N^{\Bub}$ is represented by a DM stack. Moreover, there exists a proper morphism $f: M_N^{\Bub}\to M_N$ from the moduli stack of bubble tree DM stacks to the KSBA moduli space. 

    If $N$ is sufficiently large divisible enough, $M^{\Bub}:=M_N^{\Bub}$ and $M:=M_N$ are proper DM stacks with projective coarse moduli spaces. 
\end{thm}

\section{Moduli space of stable maps}\label{sec_moduli_stable_maps}

In this section we discuss the possibility to have a virtual fundamental class on the moduli space of stable maps from slc surfaces to projective varieties. 
This idea will lead to define Gromov-Witten like invariants for counting surfaces in projective varieties. 

In general it is very hard to define a perfect obstruction theory on the moduli space of stable map space in the surface setting.  We only list some results and the details will be published in the future. 

\subsection{Kontsevich moduli space}

Let us first review the Kontsevich moduli space of stable maps from algebraic curves to projective varieties, and the perfect obstruction theory. 

 Let $W$ be a smooth projective variety or DM stack and  let $\overline{M}_{g,n}(W,\beta)$ be the moduli space or stack of stable maps in from $n$-marked curves to $W$.  Think about one point
$((C,p_1,\cdots,p_n)\stackrel{g}{\rightarrow}W)$ in $\overline{M}_{g,n}(W,\beta)$, which is a stable map from marked curves $C$ to $W$.  We have the following exact sequence of the deformation and obstruction spaces of the map $g$:
\begin{equation}\label{eqn_long_sequence_GW}
\begin{array}{cccccc}
&0& \rightarrow & T_g^0& \rightarrow & \Aut(C,p_1,\cdots,p_n)\\
\rightarrow& \Hom(g^*\Omega_W,\sO_C)& \rightarrow & T_g^1& \rightarrow & T^1(C,p_1,\cdots,p_n)\\
\rightarrow& \Ext^1(g^*\Omega_W,\sO_C)& \rightarrow & T_g^2& \rightarrow & T^2(C,p_1,\cdots,p_n)\\
\rightarrow &0
\end{array}
\end{equation}
where $T^1_g$ classifies the deformations of the map $g$, and $T^2_g$ classifies the obstructions.  Since $(C,p_1,\cdots, p_n)$ is a nodal marked curve, its obstruction $T^2(C,p_1,\cdots,p_n)=0$ vanishes. 
If the map $g$ is stable, then $T_g^0=0$.  In \cite{LT}, \cite{BF}, Behrend-Fantechi, Li-Tian studied the deformation and obstruction theory of the map $g$, and defined the perfect tangent-obstruction theory complex on the moduli space 
$\overline{M}_{g,n}(W,\beta)$.   Let $f:  \overline{M}_{g,n}(W,\beta)\to \mathfrak{M}_{g,n}$ be the morphism to the stack $\mathfrak{M}_{g,n}$ of pre-stable marked curves obtained by forgetting the map.  The stack 
$\mathfrak{M}_{g,n}$ is a smooth algebraic stack of dimension $3g-3+n$.  Look at the following diagram
$$
\xymatrix{
\sC\ar[r]^{g}\ar[d]_{\pi}& W\\
\overline{M}_{g,n}(W,\beta)
}
$$
where $\sC$ is the universal curve and $g$ is the universal map.   Behrend \cite{Behrend2} constructed a relative perfect obstruction theory 
\begin{equation}\label{eqn_relative_POT}
\phi: R^*\pi_*(g^*T_W)^{\vee}\to \ll^{\bullet}_{\overline{M}_{g,n}(W,\beta)/\mathfrak{M}_{g,n}}
\end{equation}
with the deformation and obstruction spaces are $\Hom(g^*\Omega_W,\sO_C)$ and  
$\Ext^1(g^*\Omega_W,\sO_C)$ respectively on a map $g$. 
These two perfect obstruction theories on $\overline{M}_{g,n}(W,\beta)$  give the same virtual fundamental class $[\overline{M}_{g,n}(W,\beta)]^{\vir}\in H_{2\vd}(\overline{M}_{g,n}(W,\beta))$, where $\vd$ is the virtual dimension.

\subsection{Moduli space of stable maps after V. Alexeev}

The moduli space of stable maps from slc surface pairs to projective varieties was constructed by V. Alexeev in \cite{Alexeev}. We reivew it here. 

Recall that an  slc log surface pair  $(X,D)$ is called stable if its $K_X+D$ is ample.  Let $W$ be a smooth projective scheme.   From \cite{Alexeev}, a map  $g: (X,D)\to W$ is a stable map if the following conditions are satisfied:
\begin{enumerate}
\item $D=\sum_{i=1}^{n}a_i D_i$ is a $\qq$-divisor of $X$;
\item the pair $(X,D)$ is s.l.c.;
\item the divisor $K_X+D$ is relatively $g$-ample;
\item let $H:=g^*\sO_W(1)$.  There exists an $N>0$ such that the sheaf $L_N:=\sO(N(K_X+D+5H))$ is a line bundle. 
\end{enumerate}
Let $T$ be a finite type scheme over $\mathbf{k}$.  A flat family $(\sX, \sD)/T\stackrel{g}{\rightarrow} W$ of stable maps over $T$ is given by 
flat $\qq$-Gorenstein deformation family $(\sX,\sD)\to T$ of s.l.c. log surface pairs such that $g$ is stable map on each fiber $t\in T$; and there exists an $N>0$ such that the sheaf $\sL_N:=\sO(N(K_{\sX/T}+\sD+5H))$ is a line bundle.

Fixing $K^2\in \qq$, and $I\subset [0,1]$ a finite set of coefficients of $D$ closed under addition, and $0< A, B \in\qq$. 
We define the moduli stack 
$$M:=\overline{M}_{K^2, A, B}(W)$$ 
to be the moduli functor  from the category $\Sch_{\mathbf{k}}$ of schemes over $\mathbf{k}$ to groupoids, such that  
$M(T)$ is the $\qq$-Gorenstein deformation families $\{(\sX,\sD)/T\to W\}$ of  stable maps from  s.l.c. log surface pairs to $W$
which satisfy the conditions:  for any $t\in T$, and any stable map $(\sX_t, \sD_t\stackrel{g}{\rightarrow} W)$, $(K_{\sX_t+\sD_t})^2=K^2$,  $(K_{\sX_t+\sD_t})\cdot H=A$ and 
$H^2=B$.  In addition, for sufficient large $r>0$, the line bundle $T\mapsto \det f_*\left(\sO_{\sX}(rK_{\sX/T}+\sD+5H)\right)$ associated with each family extends to a functorial line bundle on the entire 
moduli functor, where $f: (\sX,\sD)\to T$ is the family over $T$. 

\begin{thm}\label{thm_moduli_maps_intro}(\cite[Theorem]{Alexeev})
For the fixed invariants $K^2, I, A, B$ above, the moduli functor $M$ is represented by a  projective DM stack $M$ over $\mathbf{k}$.
\end{thm}

Our strategy for the enumerative geometry on the moduli space $M$ of stable maps is to consider the lci covering DM stack pairs, since it is possible that the KSBA moduli space of lci log surface pairs may admits a virtual fundamental class.

Thus, we extend the construction of the moduli stack $M$ by considering  the log DM stack pairs $(\XX, \DD)$ such that its coarse moduli space 
$(X, D)$ is an slc log surface of general type.  Examples contain index one covering DM stacks $(\XX, \DD)$, and the $\lci$ covering DM stacks
$(\XX^{\lci}, \DD^{\lci})$. 
We give an explanation here. 
An slc surface $X$ may contain certain simple elliptic and cusp singularities with embedded dimension  very high, then the higher  obstruction spaces for deforming the singularities do not vanish \cite{Jiang_2021}. 
The  $\lci$ covering DM stack $\XX^{\lci}$ of an s.l.c. surface $X$ was defined in \cite{Jiang_2022} and only has l.c.i. singularities. Thus the higher obstruction spaces vanish for  the $\lci$ covering DM stack $\XX^{\lci}$.  The log version of $\lci$ covering DM stacks can be similarly defined.

Let $\WW$ be a projective DM stack with projective coarse moduli space $W$.    We define the moduli stack of stable maps from  log surface DM stacks $(\XX, \DD)$ to $W$.  We still fix $K^2\in \qq$, $I\subset [0,1]$ a finite set, and $0< A, B\in \qq$.   Let 
$$M^{\tw}:=\overline{M}^{\tw}_{K^2, A, B}(\WW): \Sch_{\mathbf{k}}\to \text{Groupoids}$$
be the moduli functor
that sends $T$ to the isomorphism classes $\{(\XX,\DD)/T\stackrel{g}{\rightarrow}\WW\}$ of stable maps from the  log surface DM stacks  $(\XX, \DD)$ to  $\WW$ such that the 
the map $(\sX, \sD)/T\to W$ on the coarse moduli spaces is stable with the fixed invariants $K^2, A, B$ and coefficients set of $\sD$ is $I$.  We have the following result.

\begin{thm}\label{thm_index_one_DM_intro}(\cite{Jiang_Tseng}) 
The moduli functor $M^{\tw}$ is represented by a  projective DM stack $M^{\tw}$ over $\mathbf{k}$. 
Moreover, there exists a proper morphism between DM stacks 
$$f: M^{\tw}\to M$$
which induces a proper map  on the coarse moduli spaces. 
\end{thm}

\subsection{Perfect obstruction theory}

We study the obstruction theory on the moduli stack $M^{\tw}$.

Let $T=\spec(A)$ be an affine scheme, and  $(\sX,\sD)/A\stackrel{g}{\rightarrow}W$ be a family of  stable maps. 
The flat family $(\XX,\DD)/A\stackrel{g}{\rightarrow}\WW$ is one lifting of stable maps to the gerbe $\WW$. 
Let $\ll^{\bullet}_{\XX/A}$,  $\ll^{\bullet}_{\DD/A}$, and  $\ll^{\bullet}_{\WW}=\Omega_{\WW}$ be the cotangent complexes of $\XX, \DD, \WW$ respectively.
We take the stable map $g$ as the map
$$(\DD\stackrel{\iota}{\hookrightarrow}\XX)\stackrel{g}{\rightarrow} \WW.$$
Let $\overline{A}\to A$ be an extension with kernel $\mathcal{N}$, then the $\qq$-Gorenstein  deformation of the stable map $g$ over $\overline{A}$ is canonically isomorphic to 
the group 
$$\Ext^1_{\DD}(\iota^*g^*\Omega_{\WW}\to \iota^*\ll^{\bullet}_{\XX/A}\to \ll^{\bullet}_{\DD/A}, \sO_{\DD}\otimes_{A}\mathcal{N})$$
and the obstruction lies in 
$$\Ext^2_{\DD}(\iota^*g^*\Omega_{\WW}\to \iota^*\ll^{\bullet}_{\XX/A}\to \ll^{\bullet}_{\DD/A}, \sO_{\DD}\otimes_{A}\mathcal{N}).$$
For simplicity we let $T^i_{g}:=\Ext^i_{\DD}(\iota^*g^*\Omega_{\WW}\to \iota^*\ll^{\bullet}_{\XX/A}\to \ll^{\bullet}_{\DD/A}, \sO_{\DD}\otimes_{A}\mathcal{N})$. 
There may exist higher obstruction spaces $T^i_{g}$ for $i\ge 3$.  Ran, Z.'s deformation theory of maps in \cite{Ran} gives a long exact sequence for the deformations of 
$g$ as follows:
\begin{equation}\label{eqn_long_sequence_map}
\begin{array}{cccccc}
&0& \rightarrow & T_g^0& \rightarrow & \Aut(\XX,\DD)\\
\rightarrow& \Hom_g(g^*\Omega_{\WW},\sO_{(\XX,\DD)})& \rightarrow & T_g^1& \rightarrow & T^1(\XX,\DD)\\
\rightarrow& \Ext_g^1(g^*\Omega_{\WW},\sO_{(\XX,\DD)})& \rightarrow & T_g^2& \rightarrow & T^2(\XX,\DD)\\
\rightarrow& \Ext_g^2(g^*\Omega_{\WW},,\sO_{(\XX,\DD)})& \rightarrow & T_g^3& \rightarrow & T^3(\XX,\DD)\\
\rightarrow& \Ext^3(g^*\Omega_{\WW},,\sO_{(\XX,\DD)})& \rightarrow & T_g^4& \rightarrow & T^4(\XX,\DD)\\
\rightarrow&\cdots\cdots\cdots
\end{array}
\end{equation}
where 
$T^i(\XX, \DD):=\Ext^i_{\DD}(\iota^*\ll^{\bullet}_{\XX/A}\to \ll^{\bullet}_{\DD/A}, \sO_{\DD})$
and $\sO_{(\XX,\DD)}:=[\iota^*\sO_{\XX}\to \sO_{\DD}]$.   We prove the above long exact sequence in the Appendix. 
Let $T^i(\XX):=\Ext^i_{\XX}(\ll^{\bullet}_{\XX/A}, \sO_{\XX})$ and $T^i(\DD):=\Ext^i_{\DD}(\ll^{\bullet}_{\DD/A}, \sO_{\DD})$.
Then we also have a long exact sequence for the deformation of the pair $(\XX,\DD)$:
\begin{equation}\label{eqn_long_sequence_pair}
\begin{array}{cccccc}
&0& \rightarrow & T^0(\XX,\DD)& \rightarrow &T^0(\DD)\oplus T^0(\XX) \\
\rightarrow& \Hom_\iota(\iota^*\Omega_{\XX},\sO_{\DD})& \rightarrow & T^1(\XX,\DD)& \rightarrow & T^1(\DD)\oplus T^1(\XX) \\
\rightarrow& \Ext_\iota^1(\iota^*\Omega_{\XX},\sO_{\DD})& \rightarrow & T^2(\XX,\DD)& \rightarrow & T^2(\DD)\oplus T^2(\XX) \\
\rightarrow& \Ext_\iota^2(\iota^*\Omega_{\XX},\sO_{\DD})& \rightarrow & T^3(\XX,\DD)& \rightarrow & T^3(\DD)\oplus T^3(\XX) \\
\rightarrow&\cdots\cdots\cdots
\end{array}
\end{equation}

In order to get a perfect obstruction theory on the moduli space $M^{\ind}$,  it is necessary to prove that 
$T^j(\XX,\DD)=0$ for $j\ge 3$; and $\Ext_g^i(g^*\Omega_{\WW},\sO_{(\XX,\DD)})=0$ for $i\ge 2$.

\textbf{Case 1:} From \cite[Theorem 4.23, Theorem 4.24]{Kollar-Shepherd-Barron},  the  singularities of an s.l.c. surface $X$, except the normal crossing singularities in codimension one, are all  isolated singularities as follows:   finite group  quotient surface singularities,  simple elliptic singularities,  cusp singularities,  degenerate cusp singularities, $\zz_2, \zz_3, \zz_4, \zz_6$ quotients of  simple elliptic singularities,  and  $\zz_2$ quotients of cusps, and  degenerate cusps. 

From \cite[Proposition 3.10]{Kollar-Shepherd-Barron}, Koll\'ar-Shepherd-Barron  proved that if the  quotient singularities admit  $\qq$-Gorenstein smoothings, then they must be class $T$-singularities.
Therefore the index one covers of such singularities have $A_n$ type singularities which are l.c.i.   
For $\zz_2, \zz_3, \zz_4, \zz_6$ quotients $(X,x)/\zz_r$   $(r=2,3,4,6)$ of the simple elliptic singularities, and  $\zz_2$ quotients $(X,x)/\zz_2$ of cusps, and  degenerate cusps,  the index of such singularities can only be $r=2,3,4,6$ and  the index one covers are given by the germs $(X,x)$.  
Thus  for an s.l.c. surface $X$, the possible singularities of  the index one covering DM stack $\XX$ are l.c.i. singularities,  simple  elliptic singularities,  cusps, and  degenerate cusp singularities.
Also for l.c.i. singularity germs $(X,x)$, the local tangent sheaves $\sT^q(X)=0$ for $q\ge 2$. 

From \cite[Theorem 3.13]{Laufer}, a simple  elliptic singularity,  a cusp, or a degenerate cusp singularity germ $(X,x)$ who has embedded dimension $\leq 4$ is l.c.i.
Therefore, if the log surface pair $(X, D)$ does not have any simple  elliptic singularities,  cusps, and  degenerate cusp singularities, or the cyclic quotients of them, then the index one covering DM 
stack $(\XX, \DD)$ only has l.c.i. singularities.  
We make the following condition. 
\begin{condition}\label{condition_star_pair_intro}
If an s.l.c. surface $X$ has the following surface singularity $(X,x)$:   a simple elliptic singularity, a cusp or a 
degenerate cusp singularity, or the $\zz_2, \zz_3, \zz_4, \zz_6$ quotients of the simple elliptic singularity, and the $\zz_2$ quotient of a cusp or a degenerate cusp singularity, 
then $(X,x)$ has  embedded dimension 
at most $4$.  
\end{condition}

From \cite{Hacking}, a stable Hacking pair $(X, D)$ is a pair such that it is a deformation of $(\pp^2, C)$ where $C$ is a degree $d\ge 4$ hypersurface, and $-K_X$ is ample. 
Hacking stable pair $(X,D)$ satisfies the conditions above.  Thus from \cite[Proposition 4.14]{Jiang_2022}, $T^i(\XX)=0, T^i(\DD)=0$ for $i\ge 3$. 
Since the divisor $\DD$ avoids the generic codimension one singular points and $\XX$ only has l.c.i. singularities, the sheaf $\iota^*\Omega_{\XX}$ is only not locally free on some isolated singular points of $\DD$.  This will force 
$$\Ext_\iota^i(\iota^*\Omega_{\XX},\sO_{\DD})=0$$
for $i\ge 2$.

Therefore, if the log surface pair $(X,D)$ satisfies the condition \ref{condition_star_pair_intro},   from (\ref{eqn_long_sequence_pair}),  we have 
$T^j(\XX,\DD)=0$ for $j\ge 3$.

\textbf{Case 2:}
For the group $\Ext_g^i(g^*\Omega_{\WW},\sO_{(\XX,\DD)})$,  we have to restrict to {\em strongly convex} varieties. 
We call a smooth projective variety $W$ {\em strongly convex} if  whenever there is a stable map $g: (X,D)\to W$ from Hacking stable pairs, we have 
$$H^i(X, g^*T_W)=0$$
for $i=1, 2$.
For strongly convex varieties, we have $\Ext_g^i(g^*\Omega_{\WW},\sO_{(\XX,\DD)})=0$ for $i\ge 2$.  Therefore, from (\ref{eqn_long_sequence_map}), 
$T^i_g=0$ for $i\ge 3$.

Let $T=\spec(A)\in\Sch_{\mathbf{k}}$ be an affine scheme, and $\eta\in M^{\ind}(T)$ representing a complex of sheaves of $\sO_{\WW\times T}$-modules.  We let 
$$\mathcal{T}^{1}M^{\ind}(\eta)(\mathcal{N})=\sE xt^{1}_{T\times \DD/T}(\iota^*\Omega_{\WW}\to \iota^*\ll^{\bullet}_{\XX}\to \ll^{\bullet}_{\DD}, \pi_T^*\mathcal{N})$$
$$\mathcal{T}^{2}M^{\ind}(\eta)(\mathcal{N})=\mbox{ob}_{\eta}\otimes_{\sO_T}\mathcal{N},$$
where $\ob_{\eta}=o((\XX,\DD)/T\to \WW)\in \Ext^{2}_{T\times \DD/T}(\iota^*\Omega_{\WW}\to \iota^*\ll^{\bullet}_{\XX}\to \ll^{\bullet}_{\DD}, \pi_T^*\mathcal{N})$ is the obstruction class
and $\pi_T: T\times \WW\to T$ is the projection. 

We get the following result:

\begin{thm}\label{thm_tangent_obstructin_M-ind}(\cite{Jiang_Tseng})
Let $M^{\ind}$ be the moduli stack of stable maps from index one covering DM stacks to the $\mu_N$-gerbe $\WW$.  Suppose that $W$ is strongly convex, and for the log index one covering DM stack   pairs $(\XX,\DD)$ in $M^{\ind}$, the 
corresponding log surface pairs $(X,D)$ all satisfy the condition 
\ref{condition_star_pair_intro}.  Then the tangent-obstruction complex 
$$\mathcal{T}^{\bullet}M^{\ind}(\eta)(\mathcal{N})=\Big[ \sE xt^{\bullet}_{T\times \DD/T}(\iota^*\Omega_{\WW}\to \iota^*\ll^{\bullet}_{\XX}\to \ll^{\bullet}_{\DD}, \pi_T^*\mathcal{N})\Big]$$
gives a $perfect$ tangent-obstruction theory in the sense of Li-Tian in \cite[Definition 1.3]{LT}.
\end{thm}

\subsection{Discussion on the general case}

If we consider the general log surface pair $(X,D)$, and its lci covering DM stack $(\XX^{\lci}, \DD^{\lci})$. Let
$M=\overline{M}_{K^2, \chi, A,B}$ be the KSBA moduli stack.  
Suppose that $D=\sum_{i}a_i D_i$.
Then it is not known at the moment what is the right cotangent complex for this moduli space for the general coefficients $a_i$.  So in this setting it is still not clear how to define a perfect obstruction theory on the moduli stack of lci covers for this moduli space. 

Even in the cases that the log pair $(X,D)$ is a Hacking pair, or $D$ is empty, but $W$ is a general projective variety, then the moduli stack $M^{\tw}:=\overline{M}^{\tw}_{K^2, A, B}(\WW)$
admits an obstruction theory $\sT^{\bullet}(M^{\tw})$ has 
three terms. At the moment there is no construction for a virtual fundamental class for the three-term obstruction complex. 

\subsection{Enumerative invariants}

Let $M^{\ind}$ be the moduli space of stable maps in Theorem \ref{thm_tangent_obstructin_M-ind} such that it admits a perfect tangent-obstruction theory.  Standard construction in \cite{LT}, \cite{BF} gives a virtual fundamental class $[M^{\ind}]^{\vir}\in H_{2\vd}(M^{\ind})$.  Here $\vd=\dim(T_g^1)-\dim(T_g^2)$ is the virtual dimension. 

Theorem \ref{thm_index_one_DM_intro} gives a 
a finite morphism $f: M^{\ind}=\overline{M}^{\ind}_{K^2, A, B}(\WW)\to M=\overline{M}_{K^2, A, B}(W)$ between DM stacks.   We define 
$$[M, \WW]^{\vir}:=f_*([M^{\ind}]^{\vir})\in H_{2\vd}(M)$$
to be the virtual fundamental class of the moduli space $M$ associated with the $\mu_N$-gerbe $\WW\to W$. 
We do not know if the virtual fundamental class $[M, \WW]^{\vir}$ is independent to the choice of the $\mu_N$-gerbe $\WW\to W$ over $W$.  
But for the same line bundle $L\in\Pic(W)$, for a large divisible $N\in\zz_{>0}$ by the orders of the cyclic groups in the index one covering DM stacks $(\XX,\DD)$ in $M^{\ind}$, we believe  that 
$[M, \WW]^{\vir}$ is independent to the choice of $N\gg0$.   

\begin{con}\label{con_virtual_cycle_gerbes}
The virtual fundamental class $[M, \WW]^{\vir}$ constructed above is independent to the $\mu_N$-gerbes over $W$ for large divisible $N\gg 0$.
\end{con} 

Following the suggestion in \cite[\S 7]{Alexeev3}, and \cite[\S 5.3]{DR},  we define the Gromov-Witten surface counting invariants for counting Hacking pairs in $W$ using the virtual fundamental class 
$[M, \WW]^{\vir}$.  Recall in the moduli space $M^{\ind}$, the divisor $D=\sum_{i=1}^{n}a_i D_i$ in the log surface pair $(X,D)$ with fixed coefficients set $I\subset [0,1]$. 
Let $p_{ij}:=D_i\cap D_j$ be the intersection points for $i, j\in\{1, 2,\cdots, n\}$, and define 
$$\ev_{ij}: M\to W$$
by
$$((X,D)\stackrel{g}{\rightarrow}W)\mapsto g(p_{ij}).$$
\begin{defn}\label{defn_GW1}
Let $\alpha_{ij}\in H^*(W,\zz)$ be cohomology classes.  Define
$$\langle\alpha_{ij}\rangle_{K^2, A, B}^{W}:=\int_{[M,\WW]^{\vir}}\prod_{i,j} \ev_{ij}^*(\alpha_{ij}).$$
The invariant is non-zero only if $\sum_{i,j}\deg(\alpha_{ij})=2\vd$.
\end{defn}

\end{document}